\documentclass{amsart}
\usepackage{xcolor}
\usepackage{todonotes}
\usepackage{amscd,amsmath,latexsym,bm,pdfpages,verbatim,amsthm}
\usepackage{enumerate}

\usepackage[linktocpage=true, pdfencoding=auto, psdextra]{hyperref}
\usepackage{amssymb,amscd}
\usepackage[cp1251]{inputenc}
\usepackage[T2A]{fontenc}
\usepackage{graphicx}
\usepackage{pstricks,pst-grad,pst-node}
\usepackage[all]{xy}
\usepackage{dsfont}
\usepackage{bbm}
\usepackage{tikz}
\usepackage{tikz-cd}
\usetikzlibrary{cd}
\usepackage{mathtools,enumerate}
\newtheorem{thm}{Theorem}[section]
\newtheorem{cor}[thm]{Corollary}
\newtheorem{lem}[thm]{Lemma}
\newtheorem{prop}[thm]{Proposition}

\theoremstyle{definition}
\newtheorem{defin}[thm]{Definition}
\theoremstyle{definition}

\theoremstyle{remark}

\theoremstyle{remark}

\newcommand{\R}{\mathbb{R}}
\renewcommand{\C}{\mathbb{C}}
\newcommand{\Z}{\mathbb{Z}}

\def\C{{\mathbb C}}
\def\R{\mathbb{R}}

\def\z2{{\bf Z}$_{2}$}
\def\zl2{{\bf Z}$_{(2)}$}

\def\rmac{\mathcal{Z}^{\mathbb{R}}_K}
\def\mac{\mathcal{Z}_K}
\def\cxx{(\underline{CX},\underline{X})}

\def\polX{\pol{\underset{1 \leq i \leq m}{\bigoplus} \widetilde{H}^*(\Sigma X_i)}}
\def\polXI{\pol{\bigoplus \widetilde{H}^*(\Sigma X_i)}}

\renewcommand{\z}[1]{\mathcal Z_{#1}}
\newcommand{\pol}[1]{\k\langle {#1}\rangle}
\def\Tor{\mathrm{Tor}}
\def\k{\mathbbm k}
\newcommand{\ym}{y_1,\ldots,y_m}
\newcommand{\omegam}{\omega_1,\ldots,\omega_m}
\newcommand{\bideg}{\mathop{\rm bideg}}
\newcommand{\id}{\mathrm{id}}
\begin{document}

	\title[Differential Graded Models]{Models for the Cohomology of Certain Polyhedral Products}

	\author[M.~Bendersky]{M.~Bendersky}
	\address{Department of Mathematics, Hunter College, CUNY,   695 Park Avenue New York, NY 10065, U.S.A.}
	\email{mbenders@hunter.cuny.edu}
	
	\author[J.~Grbi\' c]{J.~Grbi\' c}
	\address{School of Mathematical Sciences, University of Southampton, SO17 1BJ Southampton, UK}
	\email{J.Grbic@soton.ac.uk}
	
	\begin{abstract}
	For a commutative ring $\k$ with unit, we describe and study various differential graded $\k$-modules and $\k$-algebras which are models for the cohomology of polyhedral products $(\underline{CX},\underline X)^K$. Along the way, we prove that the integral cohomology $H^*((D^1, S^0)^K; \Z)$ of the real moment-angle complex is a Tor module, the one that does not come from a geometric setting. We also reveal that the apriori different cup product structures in $H^*((D^1, S^0)^K;\Z)$ and in $H^*((D^n, S^{n-1})^K; \Z)$ for $n\geq 2$ have the same origin.
	As an application, this work sets the stage for studying the based loop space of $(\underline{CX}, \underline X)^K$ in terms of the bar construction applied to the differential graded $\Z$-algebras $B(\mathcal C^*(\underline X; \Z), K) $ quasi-isomorphic to the singular cochain algebra $\mathcal C^*((\underline{CX},\underline X)^K;\Z)$.
	\end{abstract}
	
	\
	
	\maketitle

	\tableofcontents

\section{Introduction}
	
One of the most studied algebraic invariants of topological spaces is their cohomology ring. With its rich and relatively approachable structure, cohomology is a homotopy theoretical flagship invariant harnessed by a host of mathematical disciplines. 

In the realm of toric topology, polyhedral products $(\underline X, \underline A)^K$ are constructed as a functorial interplay between topology and combinatorics in term of  topological pairs $(\underline X, \underline A)$ and a simplicial complex $K$. More precisely, let $K$ be a simplicial complex on the vertex set $[m]=\{1,\ldots, m\}$ and let $(\underline{X},\underline{A})=\{(X_i,A_i)\}_{i=1}^m$ be an $m$-tuple of $CW$-pairs. The \textit{polyhedral product} is defined by
\[
(\underline{X},\underline{A})^K = \bigcup_{\sigma \in K} (\underline{X},\underline{A})^\sigma \subseteq \prod_{i=1}^m X_i \] 
where 
\[
(\underline{X},\underline{A})^\sigma = \prod_{i=1}^m Y_i, \quad  Y_i = \begin{cases} X_i & \text{for }i \in \sigma \\ A_i & \text{for }i \notin \sigma. \end{cases}
\]
	

With $CX_i$ being the cone on $X_i$, we restrict our attention to polyhedral products where the topological pairs are $(CX_i,X_i)$ for all $i$.
Of particular importance are $X_i=S^1$ for all $i$ in which case the polyhedral product is identified with the complex {\it moment-angle complex}, and is denoted by $\mac$.  If $X_i=S^0$ for all $i$, the polyhedral product is the real {\it moment-angle complex}, denoted by $\rmac$. 

Before we start with the study of polyhedral products, let us lay down some basic combinatorial notations. For an abstract simplicial complex $K$ on the vertex set $[m]$, let $|K|$ denote its geometric realisation and let $K_J$ be the full subcomplex of $K$ consisting of all simplicies of $K$ with vertices in $J\subset[m]$.

Moment-angle complexes are known as the bellwethers for the combinatorial coding of topological and geometrical properties.	For example, there is a stable decomposition of a real moment-angle complex $\rmac=(D^1,S^0)^K$ in terms of full subcomplexes of the simplicial complex $K$ given by
\[
\Sigma (D^1, S^0)^K=\Sigma \bigvee_{J\subseteq [m]}|\Sigma K_J|.
\]

This decomposition was generalised in \cite{bbcg} to a stable decomposition of polyhedral products, which specialises to the following stable decomposition of $\cxx^K$,
\[
\Sigma \cxx^K \longrightarrow \Sigma	\bigvee_{J\subseteq [m]}|\Sigma K_J| \wedge  Y^{\wedge J}
\]
where 
\[
Y^{\wedge J} = \underset{i \in [m]}{\bigwedge }Y_i  \mbox{ with } Y_i = X_i \mbox{ if } i \in J \mbox{ and } Y_i = S^0 \mbox{ if } i \notin J.
\]
Directly from these stable decompositions, we can read off the underlying $\k$-module  of the cohomology of the real moment-angle complex, of the complex moment-angle complex and of the polyhedral product $(\underline{CX}, \underline X)^K$, respectively as
\[
H^*((D^1, S^0)^K; \k)\cong \bigoplus_{J\subseteq [m]} H^{*-1}(|K_J|;\k), \qquad H^*((D^2, S^1)^K; \k)\cong \bigoplus_{J\subseteq [m]} H^{*-|J|-1}(|K_J|;\k) \qquad \text{and}
\]
\[
  H^*((\underline{CX}, \underline X)^K; \k)\cong \bigoplus_{J\subseteq [m]}H^{\ast}(|\Sigma K_J|;\k) \otimes H^J
\]
where $H^J = \underset{i \in [m]}{\otimes }H_i$ with $H_i = \widetilde{H}^{\ast}(X_i;\k)$ if $i \in J$ and  $Y_i = \k$ if $i \notin J$ assuming that $H^*(X_i;\k)$ are free $\k$-modules for all $i$.

Let $K$ be a simplicial complex on the vertex set $[m]$. Denote by $\k[v_1,\ldots,v_m]$ the polynomial ring on $m$ variables $v_i$ and for every set $I=\{i_1,\ldots, i_r\}\subset[m]$, let $v_I=v_{i_1}\cdots v_{i_r}$. Then the {\it Stanley-Reisner} ring is defined by
\[
SR[K]=\k[v_1,\ldots,v_m]/\mathcal I_K
\]
where $\mathcal I_K=(v_I\ |\ I\notin K)$ is the ideal generated by those square-free monomials $v_I$ for which $I$ is not a simplex of~$K$.
If $v_1,\ldots, v_m$ are all of degree 2, then 
\[
H^*((\C P^\infty, *)^K;\k)\cong SR[K].
\]
Using the language of toric varieties Franz~\cite{F2} and independently using the existence of the homotopy fibration
\[
T^m=(S^1, S^1)^K\longrightarrow (D^2, S^1)^K\longrightarrow (\C P^\infty, *)^K\longrightarrow (\C P^\infty, \C P^\infty)^K=BT^m
\]
Buchstaber and Panov~\cite{bp} showed that there is an algebra isomorphism 
\[
H^*((D^2, S^1)^K;\Z)\cong \Tor_{\Z[v_1,\ldots,v_m]}(SR[K], \Z).
\]
Similarly, Franz~\cite{f} showed that there is an additive isomorphism
\[
H^*((D^1, S^0)^K;\Z_2)\cong \Tor_{\Z_2[t_1,\ldots,t_m]}(SR_{\Z_2}[K], \Z_2)
\]
where $t_i$ are of degree 1 and $SR_{\Z_2}[K]=\Z_2[t_1,\ldots,t_m]/\mathcal I_K$ which cannot be extended to a multiplicative isomorphism for the canonical product on the Tor. This additive description of the cohomology of the real moment-angle complex can be related to the homotopy fibration
\[
\Z_2^m=(S^0, S^0)^K\longrightarrow (D^1, S^0)^K\longrightarrow (\R P^\infty, *)^K\longrightarrow (\R P^\infty, \R P^\infty)^K=B\Z_2^m.
\]
In this paper we extend Franz's result by showing, as Corollary~\ref{cor:toriso}, that $H^*((D^1, S^0)^K;\k)$	is additively isomorphic to a Tor $\k$-module. However this Tor $\k$-module does not have a geometric origin and the additive isomorphism cannot be extended to a one of algebras.

Mimicking our approach for the cohomology of the real moment-angle complex, that is, by considering the polynomial ring on  suspended cohomological classes of all $X_i$'s and introducing a generalised Stanley-Reisner ring $SR(\underline X, K)$, see Definitions~\ref{def:genpoly} and~\ref{def:genSR}, we next try to describe $H^*((\underline{CX}, \underline X)^K;\k)$ as a Tor module. This approach allow us only to identify $H^*((\underline{CX}, \underline X)^K;\k)$ as a summand of such a Tor module, see Proposition~\ref{prop:CisoR}. To possibly get a description of $H^*((\underline{CX}, \underline X)^K;\k)$ as a Tor module, in future work we plan to replace the polynomial ring on suspended cohomology classes of all $X_i$'s by an algebraic object that behaves as an algebraic ``delooping" of $X_i$'s.

In Section 3 we give two integral  algebraic models $B(\underline X,K)$ and $B(\mathcal C^*(\underline X),K)$ for the 
cohomology of $(\underline{CX}, \underline X)^K$, see Propositions~\ref{Bmodelofcxx} and \ref{prop:algebramodel}.
The results are obtained by combining Franz's~\cite{f} and Cai's~\cite{c} model of the integral cohomology ring of the real moment-angle complex with the result of Bahri, Bendersky, Cohen and Gitler~\cite{bbcg1} stating that the algebra structure in $H^*((\underline{CX}, \underline X)^K;\k)$ depends only on the cup product structure in $H^*((D^1, S^0)^K; \k)$ and the cup product structure in $H^*(X_i; \k)$ for all~$i$.
The algebraic model $B(C^*(X);\k)$ is particularly interesting as it is  quasi-isomorphic to the $\k$-cochains on $(\underline{CX}, \underline X)^K$ and therefor the bar construction can be applied to it to deduce a model over $\k$ for the based loop space of $\cxx^K$ when $X_i$ have free $\k$-cohomology.

We finish the paper by explaining how seemingly different the cup product structures in $H^*((D^1, S^0)^K;\Z)$ and in $H^*((D^n, S^{n-1})^K; \Z)$ for $n\geq 2$ have the same origin.

	\section{Additive models}
	\subsection{Real moment-angle complexes }
	Let $\k$ be a commutative ring, which in this paper will be assumed to be either the ring $\Z$ of integers or a field. 
	Let $y_1,\ldots, y_m$ be of degree $1$.
	Define the graded algebra $\pol{\ym}$ as
	\[
	\pol{y_1,\ldots, y_m}= T(y_1,\ldots,y_m)/( y_iy_j=y_jy_i)
	\]
	where $T$ is a free associative algebra.
	Notice that here we are assuming commutativity not graded commutativity.  Considered as a non-graded object $\pol{y_1,\ldots, y_m}$ is isomorphic to the polynomial algebra $\k[y_1,\ldots, y_m]$. 
	Given a subset
	$I=\{i_1,\ldots,i_r\}\subset[m]$, we denote by $y_I$ the square-free monomial $y_{i_1}\cdots y_{i_r}$ in~$\pol{y_1,\ldots, y_m}$.
	
	Define an analogue of the Stanley-Reisner ring to be
	\[
	SR\langle K\rangle =\pol{y_1,\ldots, y_m}/\mathcal I_K
	\]
	where $\mathcal I_K=(y_I\ |\ I\notin K)$ is the ideal generated by those monomials $y_I$ for which $I$	is not a simplex of~$K$.
	
	For $\omega_1,\ldots,\omega_m$ of degre~$0$, define the graded algebra $L$ as
	\begin{equation}\label{Lomega}
		L(\omega_1,\ldots,\omega_m)=T(\omega_1,\ldots,\omega_m)/( \omega_i^2=0,\ \omega_i\omega_j=-\omega_j\omega_i).
	\end{equation}
	As in the case of $\pol{\ym}$, the multiplication is not graded commutative, and considered as a non-graded object $L(\omegam)$ is isomorphic to the exterior algebra $\Lambda (\omegam)$.
	
	Define the bigraded differential algebra $(E,d)$ by
	\begin{equation}\label{Ekoszul}
		E=T(\omegam,\ym)/(\omega_i^2=0,\ \omega_i\omega_j=-\omega_j\omega_i,\ y_iy_j=y_jy_i,\ \omega_iy_j=y_j\omega_i)
	\end{equation}
	such that 
	\[
		\begin{gathered}
			\bideg \omega_i=(-1,1),\quad\bideg y_i=(0,1),\\
			d\omega_i=y_i,\quad dy_i=0
		\end{gathered}
	\]
	and requiring the differential $d$ to satisfy the identity 
	\begin{equation}\label{diffsign}  d(a\cdot b)=d(a)\cdot b+(-1)^{\deg_1(a)}a\cdot d(b)\end{equation} where $\bideg(a)=(\deg_1(a), \deg_2(a))$.
	

	\begin{lem}\label{lem:koszulresolution}
		The differential graded algebra $(E,d)$ with the differential given by $d(\omega_i)=y_i$ and $d(y_i)=0$ is a free $\pol{y_1, \ldots, y_m}$-resolution of $\k$.
	\end{lem}	
	\begin{proof}
		Consider $\k$ with the $\pol{\ym}$-module structure given by the augmentation map sending each $y_i$ to zero.
	Rewrite $E$ as
		\[
		E=L(\omegam)\otimes_\k \k\langle \ym\rangle
		\]
		emphasising that $\omega_iy_j=y_j\omega_i$. Then $(E,d)$ together with the augmentation map $\varepsilon\colon E\longrightarrow\k$ defines a cochain complex of $\pol{\ym}$-modules
		\begin{multline}\label{kosz}
			0\longrightarrow L^m(\omegam)\otimes\pol{\ym}
			\stackrel{d}{\longrightarrow}\cdots\\
			\stackrel{d}{\longrightarrow} L^1(\omegam)\otimes\pol{\ym}
			\stackrel{d}{\longrightarrow}
			\pol{\ym}\stackrel{\varepsilon}{\longrightarrow}\k\longrightarrow 0
		\end{multline}
		where $L^i(\omegam)$ is the submodule of $L(\omegam)$ generated by monomials of length~$i$. We shall show that 
		$\varepsilon\colon(E,d)\longrightarrow(\k,0)$ is a quasi-isomorphism. There is an obvious inclusion $\eta\colon\k\longrightarrow E$ such that $\varepsilon\eta=\id$. To finish the proof we construct a cochain homotopy between $\id$ and $\eta\varepsilon$, that is, a set of $\k$-linear maps $s=\{s^{-i,j}\colon E^{-i,j}\longrightarrow E^{-i-1,j}\}$ satisfying the identity
		\begin{equation}\label{chaineq}
			ds+sd=\id-\eta\varepsilon.
		\end{equation}
		For $m=1$, we define the map $s_1\colon E_1^{0,*}=\pol{y}\longrightarrow E_1^{-1,*}$ by 
		\[
		s_1(a_0+a_1y+\cdots+a_jy^j)=\omega(a_1+a_2y+\cdots+a_jy^{j-1}).
		\]
		Then for $f=a_0+a_1y+\cdots+a_jy^j\in E_1^{0,*}$ we have $ds_1f=f-a_0=f-\eta\varepsilon f$ and $s_1df=0$. On the other hand, for $\omega f\in E_1^{-1,*}$ we have $s_1d(\omega f)=\omega f$ and $ds_1(\omega f)=0$. In any case~\eqref{chaineq} holds. Now we assume by induction that for $m=k-1$ the required cochain homotopy $s_{k-1}\colon E_{k-1}\longrightarrow E_{k-1}$ is already constructed. Since
		$E_k=E_{k-1}\otimes E_1$, \
		$\varepsilon_k=\varepsilon_{k-1}\otimes\varepsilon_1$ and $\eta_k=\eta_{k-1}\otimes\eta_1$, a direct calculation shows that the map
		\[
		s_k=s_{k-1}\otimes\id+\eta_{k-1}\varepsilon_{k-1}\otimes s_1
		\]
		is a cochain homotopy between $\id$ and $\eta_k\varepsilon_k$.

		Since $L^i(\omegam)\otimes\pol{\ym}$ is a free $\pol{\ym}$-module, \eqref{kosz} is a free resolution for the $\pol{\ym}$-module~$\k$. This is an analog of the Koszul resolution.
	\end{proof}
	
	Since $L(\omegam)\otimes\pol{\ym}$ is a resolution of $\k$ by free $\pol{\ym}$-modules, it follows that the cohomology of the complex 
	\[
	\Big(	L(\omegam)\otimes_\k\pol{\ym}\Big)\otimes_{\pol{\ym} }SR\langle K\rangle = L(\omegam)\otimes_{\k} SR\langle K\rangle
	\]
	is isomorphic as a $\k$-module to $\Tor_{\pol{\ym}}(SR\langle K\rangle ,\k)$. 
	
	We aim to show that there is an additive isomorphism between the cohomology $H^*((D^0,S^1)^K; \k)$ and the $\k$-module $\Tor_{\pol{\ym}}(SR\langle K\rangle ,\k)$.
	Following  Buchstaber-Panov's work~\cite{bp}, the idea is to first reduce the differential graded algebra $(E,d)$ to a finite dimensional quotient~$\bar R(K)$ without changing the cohomology. We then proceed by showing that as a $\k$-module $\bar R(K)$ is quasi-isomorphic to the underlying $\k$-module of a certain differential graded algebra $B(K)$ which in turn is quasi-isomorphic to the singular cochains of $(D^1, S^0)^K$. 
	
\begin{defin}\label{def:R}
	Let $K$ be a simplicial complex on $[m]$. Define the differential graded algebra $(\bar R(K), d)$ by
		\[
		\bar R(K) = L(\omega_1,\ldots, \omega_m) \otimes SR\langle K\rangle /(y_i^2=\omega_iy_i=0)
		\]
		with differential $d$ induced from 
		$(E,d)$, that is, $d(w_i)=y_i$ and $d(y_i)=0$.  
	\end{defin}
	
	\begin{prop}\label{prop:quotient} 
		The quotient map 
		\[
		L(\omegam) \otimes SR\langle K\rangle \longrightarrow \bar R(K)
		\]
		is an algebra quasi-isomorphism.
	\end{prop}
	
	\begin{proof} 
		There is a short exact sequence of differential graded algebras
		\[
		0\longrightarrow \mathcal I\longrightarrow \Lambda(\omegam) \otimes SR\langle K\rangle \longrightarrow \Lambda(\omegam) \otimes SR\langle K\rangle/\mathcal I\longrightarrow 0
		\]
		where $\mathcal I$ is the ideal $(y_i^2=\omega_iy_i=0)$. Specifically, $\mathcal I$ is generated by monomials which are divisible by $y_i^2$ or $\omega_iy_i$ for some $i$.
		
		We show that $H^*\mathcal I =0$. For a monomial $x\in\mathcal I$, there is a minimal index $i(x)$ such that either $y_{i(x)}^2$ or $\omega_{i(x)}y_{i(x)}$ divides $x$. Define $s\colon \mathcal I\longrightarrow\mathcal I$ on generating monomials by
		\[
		s(x)=\omega_{i(x)}\frac{x}{y_{i(x)}}.
		\]
		
		By showing that $s$ is a chain homotopy between the identity and zero, that is, $ds(x)-sd(x) = x$ the proposition statement follows.
		
		Since $x =\frac{x}{y_{i(x)}}y_{i(x)}$, we have
		\[
		ds(x)-sd(x) = d(\omega_{i(x)}\frac{x}{y_{i(x)}})-s(d(\frac{x}{y_{i(x)}})y_{i(x)})=x+\omega_{i(x)}d(\frac{x}{y_{i(x)}})-s(d(\frac{x}{y_{i(x)}})y_{i(x)}).
		\]     
		We first observe that $y^2_{i(x)}$ divides $d(\frac{x}{y_{i(x)}})y_{i(x)}$. To see this, note that either $y_{i(x)}$ or $\omega_{i(x)}$ divide $\frac{x}{y_{i(x)}}$. In either case $y_{i(x)}$ divides $d(\frac{x}{y_{i(x)}})$.
		
		We claim that for $j < i(x)$, neither $y_j^2$ nor $\omega_jy_j$ divide $d(\frac{x}{y_{i(x)}})$.
		
		Assuming the claim for the moment, we first finish the proof. Since $y_{i(x)}^2$ divides $d(\frac{x}{y_{i(x)}})y_{i(x)}$ and for smaller~$j$, both $y_j^2 $ and $\omega_jy_j$ do not divide $d(\frac{x}{y_{i(x)}})y_{i(x)}$, it follows that $s(d(\frac{x}{y_{i(x)}})y_{i(x)})=\omega_{i(x)}d(\frac{x}{y_{i(x)}})$.
		
		We next prove the claim. When $j<i(x)$, for $y_j^2$ or $\omega_jy_j$ to divide $d(\frac{x}{y_{i(x)}})$ either $y_j$ or $\omega_j$ divides $\frac{x}{y_{i(x)}}$. If $y_j$ divides $\frac{x}{y_{i(x)}}$, then $y_j^2$ cannot divide $x$ because $j < i(x)$.  So if either $y_j^2$ or $\omega_jy_j$ divides $d(\frac{x}{y_{i(x)}})$, we must have $\omega_j$ devides $\frac{x}{y_{i(x)}}$. But then neither $\omega_j$ nor $y_j$ can divide $(\frac{x}{y_{i(x)}})/\omega_j$. Thus
		\[
		d\left(\frac{x}{y_{i(x)}}\right) =d\omega_j\left[\left(\frac{x}{y_{i(x)}}\right)/\omega_j\right] =y_j\left[\left(\frac{x}{y_{i(x)}}\right)/\omega_j\right] +\omega_jd\left[\left(\frac{x}{y_{i(x)}}\right)/\omega_j\right].
		\]
		Neither summand is divisible by $y_j^2$ or $\omega_jy_j$.
	\end{proof}
	\begin{cor}
		\label{moduleiso}
		As $\k$-modules,
		\[
		\Tor_{\pol{y_1, \ldots, y_m}}(SR\langle K\rangle,\k)\cong H^{\ast}(\bar R(K)).
		\]
		\qed
	\end{cor}	
	
	Recalling that there is the bigrading of $\omega_i$ and $y_i$ given by $\bideg(w_i)=(-1,1)$ and $\bideg(y_i)=(0,1)$, we can consider $\bar R(K)$ as a bigraded differential algebra. Recall that for $I=(i_1,\ldots,i_p),  i_1 < \ldots < i_p, L=(y_{l_1},\ldots,y_{l_s}),  l_1 < \ldots < l_s $, the monomial $\omega_{i_1}\cdots\omega_{i_p}\cdot y_{l_1}\cdots y_{l_s}$ is denoted by $\omega_Iy_L$.
	
	For $0\leq p\leq m$, a $\k$-module basis for $\bar R^{-p,*}(K)$ can be given as
	\[
	\{ \omega_I y_L \, | L \in K, |I|=p, I \cap L =\emptyset\}.
	\]

	To consider $\bar R(K)$ as a differential $\k$-module, the  differential $d\colon \bar R^{*,*}(K) \longrightarrow \bar R^{*,*}(K)$, induced from its differential algebra structure, is given on the $\k$-module generators by
	\[ 
	d(\omega_I y_L) = \sum_{k=1}^p (-1)^{k+1} \omega_{i_1} \cdots \omega_{i_{k-1}}\hat{\omega}_{i_k} \omega_{i_{k+1}}\cdots\omega_{i_p} y_{l_1} \cdots y_{l_r} y_{i_k} y_{l_{r+1}}\cdots y_{l_s}
	\]
	where $l_r<i_k<l_{r+1}$. Notice that the sign is induced by the derivation identity \eqref{diffsign} and if the set $(l_1,\ldots,l_r,i_k,l_{r+1},\ldots,l_s) \notin K$, then $\omega_{i_1}\cdots \omega_{i_{k-1}}\hat{\omega}_{i_k} \omega_{i_{k+1}}\cdots \omega_{i_p} y_{l_1}\cdots y_{l_r} y_{i_k} y_{l_{r+1}}\cdots y_{l_s}=0$. 
	
	Next we relate the $\k$-module $\bar R(K)$ to the $\k$-cohomology module of the real moment-angle complex $(D^1,S^0)^K$ via the differential graded algebra model $B(K)$ of $H^*((D^1,S^0)^K;\k)$ given by Cai~\cite{c} and Franz~\cite{f}.
	
	The differential graded algebra $B(K)$ is presented  with generators $s_i$ and $t_i$ such that $\deg(s_i)=0$ and $\deg(t_i)=1$ and the relations 
	\begin{equation}\label{Brel}
	s_is_i=s_i, \quad t_is_i=t_i,\quad s_it_i=0,\quad t_it_i=0,\quad  \prod_{j\in\sigma, \sigma\notin K} t_j=0
	\end{equation}
	for all $1\leq i\leq m$ and all variables with different indices are graded commutative, that is,
	\[
	s_is_j=s_js_i, \quad t_it_j=-t_jt_i, \quad s_it_j=t_js_i \quad \text {for } i\neq j. 
	\]
	The differential $d\colon B^*(K)\longrightarrow B^*(K)$ is given by
	\[
	ds_i=-t_i, \quad dt_i=0
	\]
	and extended using the Leibniz rule $d(a\cdot b)=d(a)\cdot b +(-1)^{\deg(a)}a\cdot d(b)$. Franz~\cite{f} proved that $B(K)$ is quasi-isomorphic to  $\mathcal C^{\ast}((D^1,S^0)^K,\k))$, the singular $\k$-cochains of the real moment-angle complex $(D^1,S^0)^K$. 
	
	We treat $B(K)$ as a bigraded differential algebra by setting that $\bideg(s_i)=(-1,1)$ and $\bideg(t_i)=(0,1)$ for all $1\leq i\leq m$.	
	Then the $\k$-module basis for $B^{-p,*}(K), 0\leq p\leq m$ is similar to the one for $\bar R^{-p,*}(K)$ and is given by 
	\[
	\{ s_I t_L\, |\,  L \in K,\, I \cap L =\emptyset\} 
	\]
	where $s_I=s_{i_1}\cdots s_{i_p}$ for $I=(i_1,\ldots,i_p)$ and $t_L =t_{l_1}\cdots t_{l_s}$ for $L=(l_1,\ldots,l_s)$.
	The $\k$-module differential $d\colon B^{*,*}(K)\longrightarrow B^{*,*}(K)$ is given by
	\[
	d(s_I t_J) = \sum_{k=1}^p(-1)^{r+1} s_{i_1}\cdots s_{i_{k-1}} s_{i_{k+1}}\cdots s_{i_p} t_{l_1}\cdots t_{l_r} t_{i_k} t_{l_{r+1}}\cdots t_{l_s}
	\]
	where $l_r<i_k<l_{r+1}$. Notice that in $B(K)$ the summand $s_{i_1} \cdots s_{i_{k-1}}s_{i_{k+1}}\cdots s_{i_p} t_{l_1}\cdots t_{l_r} t_{i_k} t_{l_{r+1}}\cdots t_{l_s}=0$ if  $(l_1,\ldots,l_r,i_k,l_{r+1},\ldots,l_s) \notin K$.  The graded-commutative Leibniz   contributes $-1$ as $\deg(s_i)=0$ and $d(s_{i_k})=-t_{i_k}$.  The additional $(-1)^r$ in the differential formula comes about from $t_{l_k}$ passing $r$ many  $t_j$'s.
	
	\begin{defin}\label{def:quasi}
		Let an additive isomorphism $f\colon\bar R^{-p,*}(K) \longrightarrow B^{-p,*}(K)$ be defined by
		\[
		f(\omega_I y_L) = \epsilon(I,L) s_I t_L
		\]
		where $\epsilon(I,L)$ is the sign of the permutation that converts $IL$,  the concatenation of $I$ followed by $L$, into an increasing sequence.
	\end{defin}
	
	\begin{prop}
	\label{prop:fddf}
	The $\k$-isomorphism $f$ commutes up to sign with the differentials.  Specifically, if $\alpha \in \bar R^{-p,*}(K)$ then 
		\[	
		f d_{\bar R}(\alpha)=  (-1)^{p} d_{B} f(\alpha).
		\] 

	\end{prop} 
	\begin{proof}
		For $\alpha = \omega_Iy_L$,
		\[
		d(\alpha)- \Sigma (-1)^{k+1} \omega_{i_1} \cdots \omega_{i_{k-1}} \omega_{i_{k+1}} \cdots \omega_{i_p} y_{l_1} \cdots y_{l_r} y_{i_k} y_{l_{r+1}}\cdots y_{l_s}.
		\]

		We compute the sign of the permutation that converts 
		\begin{equation}
			\label{permute}
			i_1,\ldots, i_{k-1}, i_{k+1}, \ldots, i_p,l_1,\ldots, l_r,i_k,l_{r+1},\ldots,l_s
		\end{equation}
		into an increasing sequence.
		
		To start, note that the sequence of $i's$ and the sequence of $l's$ are increasing. Therefore  it suffices to  permute the $l_{j}'s$ to the left starting with  $l_{1}$ followed by increasing $l_{j}'s$.
		
		The sign of the permutation that converts \eqref{permute} into an increasing sequences differs from $\epsilon(I,L)$ in two ways.

		First, we move the indices $l_1,\ldots,l_r$ to the left.  Note that $l_{j} < i_k$ for $j \leq r$. These $l_{j}'s$ are meshed into $i_1,\ldots,i_{k-1}$.  The number of permutations needed to place $l_{j}$ into its final position differs by $-1$ from the number needed to order $IL$ because $i_k$ is missing.  So moving $l_1,\ldots,l_r$ differs from $\epsilon(I,L)$ by $(-1)^r$.
		
		Second, we have to move $i_k$.  Moving this index contributes $(-1)^{p-k}$.

		Thus the sign of the permutation that converts \eqref{permute} into an increasing sequence differs from $\epsilon(I,L)$ by $(-1)^{r+k+p}$.  Therefore the  map $f$ multiplies the $k$th term in $d_{\bar R}(\alpha)$ by $(-1)^{r+k+p}\epsilon(I,L)$ which is $(-1)^{p}\epsilon(I,L)$ times  the coefficient of the $k$-th term in $d_{B}$.
	\end{proof}

	\begin{cor}\label{cor:toriso}
		The chain map $f$ induces a $\k$-module isomorphism
		\[
		\Tor_{\pol{y_1,\ldots,y_m}}(SR\langle K\rangle,\k)\longrightarrow H^*(B(K))\cong H^*((D^1,S^0)^K;\k).
		\]\qed
	\end{cor}

	\subsection{A generalisation to polyhedral products}\label{sec:cxx}
In many ways, a combinatorial contribution of the simplicial complex $K$ to homotopy theoretical properties of the real and complex moment-angle complexes is comparable and usually suggests a formulation of those properties for polyhedral products $\cxx^K$, see for example~\cite{aj, bg}. Having that both cohomology $\k$-modules $H^*((D^2, S^1);\k)$ and $H^*(D^1, S^0); \k)$ can be expressed in terms of appropriate Tor $\k$-modules, we generalise the previous discussion to give an additive description of $H^{\ast}(\cxx^K;\k)$ as a direct summand of $\mathrm{Tor}_A(M,\k)$, where $A$ and $M$ are defined below.

	
Throughout the remainder of the paper, when the cohomology $H^{\ast}(\cxx^K;\k)$ of the polyhedral product $\cxx^K$ is considered we assume that $H^{\ast}(X_i;\k)$ are free $\k$-modules for all $i$. If not explicitly stated all cohomology groups of topological spaces are taken with $\k$ coefficients.

Bahri, Bendersky, Cohen and Gitler~\cite{bbcg} gave a stable homotopy decomposition of $\cxx^K$ by establishing the homotopy equivalence
\begin{equation}\label{eqn:splitting2}
\Sigma \cxx^K \longrightarrow \Sigma	\bigvee_{J\subseteq [m]}|\Sigma K_J| \wedge  Y^{\wedge J}
\end{equation}
where 
\[
Y^{\wedge J} = \underset{i \in [m]}{\bigwedge }Y_i  \mbox{ with } Y_i = X_i \mbox{ if } i \in J \mbox{ and } Y_i = S^0 \mbox{ if } i \notin J.
\]
They further showed~\cite{bbcg1} that the map
\begin{equation}\label{eqn:splittingH}
			H^{\ast}(\cxx^K;\k) \longrightarrow 
			\bigoplus_{J\subseteq [m]}H^{\ast}(|\Sigma K_J|;\k) \otimes H^J
		\end{equation}
where $H^J = \underset{i \in [m]}{\otimes }H_i$ with $H_i = \widetilde{H}^{\ast}(X_i;\k)$ if $i \in J$ and  $H_i = \k$ if $i \notin J$
is an isomorphism of rings.
This specialises to an isomorphism
\begin{equation}\label{rmac}   H^{\ast}((D^1,S^0)^K;k) \cong 	\bigoplus_{J\subseteq [m]}H^{\ast}(|\Sigma K_J|;\k).
\end{equation}
It is proven in \cite{bbcg1} that the cup product on $H^{\ast}((D^1,S^0)^K;\k) $ restricts to a pairing
\[
H^{\ast}(|\Sigma K_J|;\k) \otimes H^{\ast}(|\Sigma K_L|;\k) \to H^{\ast}(|\Sigma K_{J\cup L}|;\k).
\]
This in turn induces a product on
\begin{equation}
\label{starprod}
\bigoplus_{J\subseteq [m]}H^{\ast}(|\Sigma K_J|;\k) \otimes H^J
\end{equation}
which is called the {\em $\ast-$product} in \cite{bbcg1}.

To give a $\k$-module model for $H^{\ast}(\cxx^K;\k)$, we shall only use the underlying additive isomorphism in \eqref{eqn:splittingH}.  In Section~\ref{sectionAlgebra} we shall invoke the multiplicative structure.  
	
By Corollaries~\ref{moduleiso} and \ref{cor:toriso},  the cohomology $\k$-module structure of $(D^1,S^0)^K$ is also modeled by the differential graded algebra $\bar R(K)$.  We next identify the summands in \eqref{rmac} with sub dga's of $\bar R(K)$.

\begin{defin}\label{def:supp}
For an element $\omega_I y_L \in \bar R(K)$,  define its {\em support} as
$$\mathrm{supp}(\omega_I y_L)=I\cup L.$$
The differential sub-module of $\bar R(K)$ generated by monomials with support $J$ is denoted by $\bar R_J(K)$.
\end{defin}
	
The following lemma is a direct consequence of  \cite{f} and \cite[Lemma 4.5.1]{bp}.
	
\begin{lem}\label{lem:support}
There is an algebra isomorphism
\[
H^{\ast}(\bar  R_J(K); \k) \cong H^{\ast}(\Sigma |K_J|); \k).
\]\qed
\end{lem}

\begin{defin}
Let $(C(\underline{X}, K), d)$ be a differential graded $\k$-module defined by
\[
C(\underline{X}, K)=\bigoplus_{J\subset [m]} \bar R_J(K)\otimes H^J
\]
where $H^J = \underset{i \in [m]}{\bigotimes }H_i$ and  $H_i = \widetilde{H}^{\ast}(X_i;\k)$ if $i \in J$ and $H_i =\k$ if $i \notin J$. 

The differential $d\colon C^*(\underline{X}, K)\longrightarrow C^*(\underline{X}, K)$ is given as the tensor product of the differential on $\bar R_J(K)$, induced from the dga $\bar R(K)$, and the trivial differential on $H^J$, induced by the trivial differential on $H^*(X_i)$.
\end{defin}

Combining Lemma~\ref{lem:support} with \eqref{eqn:splittingH}, the following statement holds.
	
\begin{lem}\label{lem:splitdga} There is a $\k$-module isomorphism 
\[
H^*(\cxx^K);\k) \longrightarrow H^*(C(\underline{X}, K)).
\]\qed
\end{lem}

To integrate the topological structure of the polyhedral product $\cxx^K$ into our cohomology model, we start by generalising the polynomial algebra $\pol{\ym}$ to a free commutative algebra on the cohomology of $\Sigma \underline{X}$.
	
\begin{defin}\label{def:genpoly}
For $CW$-complexes $X_i, 1\leq i\leq m$, let
\begin{equation}\label{genpoly}
\pol{\bigoplus_{1 \leq i \leq m} \widetilde{H}^*(\Sigma X_i)} = T(\bigoplus_{1 \leq i \leq m}\widetilde{H}^*(\Sigma X_i))/(\alpha \beta = \beta \alpha \, |\,  \alpha, \beta \in \bigoplus_{1 \leq i \leq m} \widetilde{H}^*(\Sigma X_i))
\end{equation}
where $T(M)$ denotes the free associative algebra generated by a free $\k$-module $M$. 
\end{defin} 

To algebra~\ref{genpoly} we associate a generalised Stanley-Reisner ring.
	
\begin{defin}\label{def:genSR}
For $CW$-complexes $X_i, 1\leq i\leq m$, define the generalised Stanley-Reisner ring $SR(\underline X, K)$ as
\[
SR(\underline{X}, K) = \polX/\mathcal{I}_K
\]
where $\mathcal{I}_K$ is the ideal generated by square free monomial
\[
\alpha_{i_1}\cdots \alpha_{i_t}, \text{ where }\alpha_{i_j} \in \widetilde{H}^{\ast}(X_{i_j}) \text{ and } \{i_1,\cdots , i_t\} \notin K.
\]		
\end{defin} 
	
We prove that $ H^*(C(\underline{X}, K))$ additively splits off $\mathrm{Tor}_ {\polXI}(SR(\underline{X}, K),\k)$.  To this end we choose an ordered bases $B_i$ for $\widetilde{H}^{\ast}(\Sigma X_i)$,
\[
B_i=\{b_{i,1}, \cdots, b_{i,k_i}\}
\]
which induces an ordering on $\underset{i}{\bigoplus} \widetilde{H}^{\ast}(\Sigma X_i)$ by saying that  $b<b^{\prime}$ if $b\in H^{\ast}(\Sigma X_i), b^{\prime} \in H^{\ast}(\Sigma X_j)$ and $i<j$. To make $\underset{i}{\bigoplus} \widetilde{H}^{\ast}(\Sigma X_i)$ into a bigraded object, let $\bideg(b)=(0,|b|)$.
	
We define $(L(\widetilde H^*(\underline X)),d)$ and $(E(\widetilde H^*(\underline X)),d)$, the natural generalisation of \eqref{Lomega} and \eqref{Ekoszul}, as
\[
L(\widetilde{H}^*(\underline X))= L(\underset{i}{\oplus} uB_i) = T( \oplus uB_i )/(ub_i^2=0,\ ub_iub_j=-ub_iub_j) 
\]
\[
E((\widetilde{H}^*(\underline X))=L(\widetilde{H}^*(\underline X))\otimes \pol{\bigoplus_{1 \leq i \leq m} \widetilde{H}^*(\Sigma X_i)}
\]
where $uB_i$ is generated by classes $ub$ of bidegree $(-1,|b|)$ corresponding to $b\in B_i$.  Define the differential $d$ on  $E(\widetilde H^*(\underline X))$ by $d(ub)=b$, $d(b)=0$ and by requiring that $d$ satisfies the Leibniz identity $d(a\cdot b)=d(a)\cdot b+(-1)^{|\deg_1(a)|}a\cdot d(b)$, where $\bideg(a)=(\deg_1(a), \deg_2(a))$.  As in Lemma \ref{lem:koszulresolution}, 	$E(\widetilde{H}^*(\underline X))$ is a resolution of $\k$.  Consequently the cohomology of the  Koszul complex 
\[
L(\underset{i}{\bigoplus} uB_i) \otimes  SR(\underline{X}, K)
\]
is $\mathrm{Tor}_ {\polXI}(SR(\underline{X}, K),\k)$.

\begin{defin}Let $K$ be a simplicial complex on $[m]$ and let $\underline X=\{X_i\}_{i=1}^m$ be $CW$-complexes. Define the differential graded algebra $(R(\underline X,K), d)$ by
\[
R(\underline X,K) = L(\underset{i}{\oplus} uB_i) \otimes SR(\underline{X}, K)/(b_i^2=(ub_i)b_i=0)
\]
with differential $d$ induced from the differential graded algebra $(E(\widetilde H^*(\underline X)),d)$.
\end{defin}	

Following the lines of the proof of Proposition~\ref{prop:quotient}, we established that the finite differential graded algebra $R(\underline X,K)$ is a model for the Tor $\k$-module $\mathrm{Tor}_ {\polXI}(SR(\underline{X}, K),\k)$.
\begin{lem}\label{lem:RXisTor}
The quotient map 
\[
L(\underset{i}{\oplus} uB_i) \otimes  SR(\underline{X}, K) \longrightarrow R(\underline X,K)
\]
is an algebra quasi-isomorphism. \qed
\end{lem}  	
	
To describe an additive basis of $R(\underline X,K)$ we note that $\oplus uB_i$ inherits an ordering from the ordering of $\oplus B_i$.  	A basis for $R(\underline X,K)$ is given by
\begin{equation}\label{basisR}
\{ ub_{i_1,k_1}\cdots ub_{i_s,k_s} b_{l_1, j_1} \cdots b_{l_t, j_t}\}
\end{equation}	
where
\begin{enumerate}[(i)] 
		\item if $ub$ is a factor, then $b$ is not a factor;	
		\item  $ub_{i_1,k_1}<\cdots < ub_{i_s,k_s}$,\quad    $b_{l_1,j_1}< \cdots <b_{l_t,j_t}$; 
		\item  $\{l_1,\cdots, l_t\} \in K$.
\end{enumerate}

There is a similar basis for $C(\underline{X},K)$,

\begin{equation}\label{basisC}\left\{ \omega_{i_1} \cdots \omega_{i_s} y_{l_1}\cdots y_{l_t} \otimes \Big[s^{-1}b_{i_1,k_1}\cdots s^{-1}b_{i_s,k_s} \Big] \otimes \Big[s^{-1} b_{l_1, j_1} \cdots s^{-1} b_{l_t, j_t}\Big] \right\}.\end{equation}
The difference between the basis \eqref{basisR} and basis \eqref{basisC} is that the integers $\{i_1, \cdots i_s, l_1,\cdots l_t\}$ are all distinct in \eqref{basisC}.
	
We now compare the differential graded algebras $R(\underline X,K)$ and $C(\underline X, K)$.
	
\begin{prop}\label{prop:CisoR} 
The cohomology $H^*(\cxx^K;\k)$, seen as a $\k$-module, is a direct summand of the $\k$-module $\mathrm{Tor}_ {\polXI}(SR(\underline{X}, K),\k)$.
\end{prop}
\begin{proof}	
There are maps of differential $\k$-modules:
\[
h\colon C(\underline X, K)\longrightarrow R(\underline X,K)
\]
and 
\[
g\colon R(\underline X,K) \longrightarrow C(\underline X, K)
\]
given by
\[
h\Big( \omega_{i_1} \cdots \omega_{i_s} y_{l_1}\cdots y_{l_t} \otimes \Big[s^{-1}b_{i_1,k_1}\cdots s^{-1}b_{i_s,k_s} \Big]\otimes \Big[s^{-1} b_{l_1, j_1} \cdots s^{-1} b_{l_t, j_t}\Big] \Big)=(ub_{i_1,k_1}\cdots ub_{i_s,k_s})( b_{l_1, j_1}\cdots  b_{l_t, j_t})
\]
and 
\[
g\Big((ub_{i_1,k_1}\cdots ub_{i_s,k_s})( b_{l_1, j_1}\cdots  b_{l_t, j_t})\Big)=
 \omega_{i_1} \cdots \omega_{i_s} y_{l_1}\cdots y_{l_t} \otimes \Big[s^{-1}b_{i_1,k_1}\cdots s^{-1}b_{i_s,k_s} \Big]\otimes \Big[s^{-1} b_{l_1, j_1} \cdots s^{-1} b_{l_t, j_t}\Big]. 
\] 
	
The homomorphisms $h$ and $g$ commute with differentials and 
$g \circ h = id$.	
\end{proof}

We note that $H^*(\cxx^K;\k) \cong   \mathrm{Tor}_ {\polXI}(SR(\underline{X}, K),\k)$ if $X_i = S^{n_i}$ for each $i$.  For $\widetilde{H}^*(X_i)$ with more than one generator, the map $g$ is zero on generators, $(ub_{i_1,k_1}\cdots ub_{i_s,k_s})( b_{l_1, j_1}\cdots  b_{l_t, j_t})$ whenever  a repetition occurs in the list $$i_1, \ldots,i_s, l_1,\ldots, l_t$$  obtained by dropping the second subscripts in the bi-indexing.  
	
	
	\section{Algebra models}\label{sectionAlgebra}
	
In this section we coalesce the Bahri-Bendersky-Cohen-Gitler $*-$product~\eqref{starprod} on $H^*(\cxx^K)$ with the Cai~\cite{c} and Franz~\cite{f} differential algebra $B(K)$ to give a natural differential algebra $B(\underline{X},K)$ whose cohomology is isomorphic the the cohomology of $\cxx^K$.

\begin{defin}\label{def:BXK}
Let $K$ be a simplicial complex on $[m]$ and let $\underline X=\{X_i\}_{i=1}^m$ be $CW$-complexes. Define a differential bigraded non-commutative algebra $(B(\underline X, K), d)$ as
\[
B(\underline X, K)=\bigoplus_{J\subset [m]} B_J(K) \otimes H^J
\]
where $B_J(K)$ is a subalgebra of $B(K)$ consisting of elements with support $J$ and $H^J=\underset{i \in [m]}{\bigotimes }H_i$ with  $H_i=\widetilde{H}^{\ast}(X_i;\k)$ if $i \in J$ and $H_i =\k$ if $i \notin J$. 

The differential $d$ on $B(\underline X, K)$ is given as the tensor product of the differential on $B_J(K)$,
induced from the dga $B(K)$, and the trivial differential on $H^J$, induced by the trivial differential on $H^*(X_i)$.
\end{defin}

We now recognise $B(\underline X, K)$ as a dga model for the cohomology of $(\underline{CX}, \underline X)^K$.

\begin{prop}
\label{Bmodelofcxx}
Let $K$ be a simplicial complex on $[m]$ and let $\underline X=\{X_i\}_{i=1}^m$ be $CW$-complexes with $H^*(X_i)$ being free $\k$-modules. Then the dga $B(\underline X, K)$ is quasi-isomorphic to $H^*(\cxx^K; \k)$.
\end{prop}
\begin{proof}
The statement follows as a direct consequence of the Bhari-Bendersky-Cohen-Gitler description of the cup product structure on $H^*(\cxx^K;\k)$ given by the $*-$product~\ref{starprod},  the K\"{u}nneth theorem and the result of Franz~\cite{f} stating that the dga $B(K)$ is quasi-isomorphic to $\mathcal C^{\ast}( (D^1,S^0)^K)$.
\end{proof}

We next  enhance the construction to give a dga which is quasi-isomorphic to
$\mathcal C^{\ast}(\cxx^K, \k)$, the singular cochains of $\cxx^K$.   

\begin{defin}\label{def:BCXK}
Let $K$ be a simplicial complex on vertex set $[m]$ and let $\underline X=\{X_i\}_{i=1}^m$ be $CW$-complexes. Define a differential bigraded non-commutative algebra $(B(\mathcal C^*(\underline X), K), d)$ as
\[
B(\mathcal C^*(\underline X), K)=\bigoplus_{J\subset [m]} B_J(K) \otimes \mathcal C^J
\]
where $B_J(K)$ is a subalgebra of $B(K)$ consisting of elements with support $J$ and $\mathcal C^J=\underset{i \in [m]}{\bigotimes }\mathcal C_i$ with  $\mathcal C_i=\mathcal C^{\ast}(X_i,\k)$ if $i \in J$ and $\mathcal C_i =\k$ if $i \notin J$. 

The differential $d$ on $B(\mathcal C^*(\underline X), K)$ is given as the tensor product of the differential on $B_J(K)$,
induced from the dga $B(K)$, and the differential on $\mathcal C^J$, induced by the differential on $\mathcal C^*(X_i,\k)$.
\end{defin}

\begin{prop}
\label{prop:algebramodel}
Let $K$ be a simplicial complex on $[m]$ and let $\underline X=\{X_i\}_{i=1}^m$ be $CW$-complexes with $H^*(X_i)$ being free $\k$-modules. Then the dga $\mathcal C^*((C\underline X, \underline X)^K)$ is quasi-isomorphic to the dga $B(\mathcal C^{\ast}(\underline{X}),K)$.\qed
	\end{prop}
\begin{proof}
The statement is a straightforward consequence of Proposition~\ref{Bmodelofcxx} and the K\" unneth formula.
\end{proof}
The bar construction applied to $B(\mathcal C^{\ast}(\underline{X}),K)$ gives a model for the loops space of $\cxx^K$ when spaces $X_i$ have torsion free $\k$-cohomology.  We will return to this point in a future paper.

We finish the paper by compering the dga $C(\underline X, K)$ with the dga $B(\underline X, K)$.
Restricting to the $\k$-module structures, we observe that $B(\underline X, K)$ is a differential bigraded $\k$-module model for $H^*(\cxx^K;\k)$. 

By extending the map $f\colon \bar R(K)\longrightarrow B(K)$ in Definition \ref{def:quasi}, we define
 an isomorphism of differential $\k$-modules
$f_{\underline X}\colon C(\underline{X},K)\longrightarrow B(\underline{X},K)$ by 
\[
	f_{\underline X}(\omega_I y_L\otimes h) = \epsilon(I,L) s_I t_L\otimes h
\]
where $\epsilon(I,L)$ is the sign of the permutation that converts $IL$,  the concatenation of $I$ followed by $L$, into an increasing sequence.

Straightforwardly, following the proof of Proposition~\ref{prop:fddf} we have the following statement.
\begin{lem}
The additive isomorphism $f_{\underline X}$ commutes up to sign with the differentials. Specifically, 
\[
f_{\underline{X}}d_{C(\underline X, K)}(\omega_Iy_L\otimes h)= (-1)^{|I|}d_{B(\underline X, K)}f_{\underline X}(\omega_Iy_L\otimes h).
\]\qed
\end{lem}
There is a natural algebra structure on $C(\underline{X}, K)=\bigoplus_{J\subset [m]} \bar R_J(K)\otimes H^J$ induced by the algebra structures on $\bar R(K)$ and $H^*(X_i)$ for all $i$. Notice that $\bar R(K)$ is a commutative algebra while $B(K)$ is not commutative. Therefore, although  $B(\underline X, K)$ and $C(\underline X,K)$ are isomorphic as  $\k$-modules, the isomorphism cannot be extended to the one of  algebras. However, in the case when all $X_i$'s are suspension spaces the algebra structure of $B(\underline X, K)$ reduces to the one of $ C(\underline X, K)$.

\begin{prop}
The algebra structures on $B(\underline X, K)$ and $C(\underline{X}, K)$ coincide  up to sign on the product of classes with disjoint support.  In particular, the algebra structures are isomorphic  up to sign if $X_i$ is a suspension space for all $i$.
\end{prop}
\begin{proof}
It is enough to see that dgas $\bar R(K)$ and $B(K)$ differ only on elements with repeated indices. Notice that in Definition~\ref{def:R} of $\bar R(K)$, we quotient out $y_i^2$ and $\omega_iy_i$. Therefore elements in which indices are repeated are trivial.
On the other hand, the defining relations~\eqref{Brel} of $B(K)$ state that the products with repeated indices, such as $s_is_i=s_i$ and $t_is_i=t_i$, are not trivial. Moreover, together with $s_it_i=0$, these relations imply that $B(K)$ is a non-commutative dga.

If we assume that $X_i$'s are suspension spaces, then since the diagonal map on suspension spaces is null-homotopic, the cup product on $H^J$ is trivial if indices are repeated. This trivialises products with repeated indices in $B(\underline X, K)$ as well despite the product being non-trivial in $B(K)$.
\end{proof}


	

\bibliographystyle{amsalpha}

\end{document}